\documentclass[11pt,letterpaper]{amsart}
\usepackage{geometry}
\usepackage[authoryear,sort&compress]{natbib}
\bibpunct{[}{]}{;}{n}{,}{,}

\usepackage{amsmath}
\usepackage{amsfonts}
\usepackage{amssymb}
\usepackage{comment}
\usepackage[shortlabels]{enumitem}
\usepackage{mathrsfs}
\usepackage{amsthm}
\usepackage{tikz}
    \usetikzlibrary{matrix}
    \usetikzlibrary{decorations.markings}
\usepackage{hyperref}
\usepackage[T1]{fontenc}
\usepackage{oldgerm}
\usepackage{scalerel,stackengine}
\usepackage{stmaryrd}
\usepackage{xcolor}
\usepackage{xfrac}
\usepackage{MnSymbol}

\newcommand{\ad}{\textrm{\normalfont ad}}
\newcommand{\Ad}{\textrm{\normalfont Ad}}

\newtheorem{theorem}{Theorem}[section]
\newtheorem{definition}{Definition}[section]
\newtheorem{example}{Example}[section]
\newtheorem{proposition}{Proposition}[section]
\newtheorem{lemma}{Lemma}[section]
\newtheorem{corollary}{Corollary}[section]
\newtheorem{remark}{Remark}[section]

\begin{document}

\title[]{Higher order jet bundles of Lie group-valued functions}

\author{Marco Castrillón López}
\address{Facultad de Ciencias Matemáticas, Universidad Complutense de Madrid, Plaza de las Ciencias, 3, Madrid, 28040, Madrid, Spain.}
\email{mcastri@mat.ucm.es}

\author{\'Alvaro Rodr\'iguez Abella}
\address{Instituto de Ciencias Matem\'aticas (CSIC-UAM-UC3M-UCM), Calle Nicol\'as Cabrera, 13-15, Madrid, 28049, Madrid, Spain.}
\email{alvrod06@ucm.es}


\begin{abstract}
For each positive integer $k$, the bundle of $k$-jets of functions from a smooth manifold, $X$, to a Lie group, $G$, is denoted by $J^k(X,G)$ and it is canonically endowed with a Lie groupoid structure over $X$. In this work, we utilize a linear connection to trivialize this bundle, i.e., to build an injective bundle morphism from $J^k(X,G)$ into a vector bundle over $G$. Afterwards, we give the explicit expression of the groupoid multiplication on the trivialized space, as well as the formula for the inverse element. In the last section, a coordinated chart on $X$ is considered and the local expression of the trivialization is computed.
\end{abstract}

\keywords{fiber bundle, Lie groupoid, jet bundle, partition, tensor product} 
\subjclass{22E30, 58A20, 22E60}

\maketitle


\section{Introduction}

For each $k\in\mathbb Z^+$, the family of $k$-jets of smooth curves, $c:\mathbb R\to G$, where $G$ is a Lie group, is naturally a Lie group fiber bundle over $\mathbb R$. We denote it by $J^k(\mathbb R,G)$, and its fiber over each $t\in\mathbb R$ by $J_t^k(\mathbb R,G)$. In fact, these spaces are frequently known as higher-order tangent spaces and denoted by $T^kG$. In \cite{Vi2013}, the Lie group structure of $J_0^k(\mathbb R,G)$ was investigated. Namely, the space was trivialized by the right, yielding an isomorphism of bundles over $G$,
\begin{equation*}
J_0^k(\mathbb R,G)\simeq G\times(\mathfrak g\oplus\overset{\underset{\smile}{k}}{\dots}\oplus\mathfrak g),
\end{equation*}
where $\mathfrak g$ is the Lie algebra of $G$. Following that, the explicit expression of the multiplication, as well as the inverse element, were studied under this identification.

The aim of this work is to generalize the results in \cite{Vi2013} when an arbitrary smooth manifold, $X$, is considered instead of $\mathbb R$. As we will see, a linear connection on the cotangent bundle of $X$ is necessary in order to trivialize the jet bundle, $J^k(X,G)$.  It is important to point out that the generalization analyzed here goes beyond a mere complications of the computations occurring in higher order tangent bundles of Lie groups. The situation for $\mathrm{dim}X\geq 2$ implies qualitative and important differences that entail the structure jet bundles in their full complexity. In particular, in this case the trivialization is not an isomorphism any more, but only an injective morphism. Notwithstanding, we take advantage of this injection to compute explicit formulas for the fibered multiplication and the inverse element.

Jets of curves on Lie groups are essential in the analysis of higher order Lagrangian mechanical systems over Lie groups \cite{EsKuSu2021}, as well as in the higher order Euler--Poincaré reduction procedure \cite{GaHoMeRaVi2012} and its applications to optimal control theory \cite{CoMa2014}. For that reason, we expect that our generalization to jets of functions defined on arbitrary manifolds will have great relevance in the field theoretical counterpart of such systems, i.e., higher order Lagrangian field theories on principal bundles and the corresponding Euler--Poincaré field equations.

The paper is organized as follows. In Section \ref{sec:prelim}, some notions about jet bundles, Lie algebras and anti-lexicographically partitions are recalled. Besides, partly ordered partitions are introduced and some properties about them are proved. Next, we present the main results of the paper: the right trivialization of $J^k(X,G)$ is defined in Section \ref{sec:trivalization}, and the fibered product, together with the formula for the inverse element, are computed under this trivialization in Section \ref{sec:structure}. Lastly, in Section \ref{sec:coordinated}, the local expression of the trivialization is given after choosing a coordinated chart of $X$. In particular, this enables us to compute the image of the injective morphism for $k=2$.

In the following, every manifold or map is assumed to be smooth, meaning $C^\infty$, unless otherwise stated. We assume that $\dim X=n$. The space of (smooth) sections of a fiber bundle, $\pi_{Y,X}:Y\to X$, is denoted by $\Gamma(\pi_{Y,X})=\Gamma(Y\to X)=\Gamma(Y)$. In particular, vector fields on a manifold $X$ are denoted by $\mathfrak X(X)=\Gamma(TX)$, where $TX$ is the tangent bundle of $X$. Likewise, $p$-forms on $X$ are denoted by $\Omega^p(X)=\Gamma(T^*X)$, where $T^*X$ is the cotangent bundle of $X$. The derivative, or tangent map, of a map $f\in C^\infty(X,X')$ between the manifolds $X$ and $X'$ is denoted by $(df)_x:T_xX\rightarrow T_{f(x)}X'$ for each $x\in X$. When working in local coordinates, we will assume the Einstein summation convention for repeated indices.

\section{Preliminaries}\label{sec:prelim}

\subsection{Jet bundles}

We summarize the notation on jet bundles that we will use in the following (for the corresponding definitions, the author can go, for example, to \cite{Sa1989}). Let $\pi_{Y,X}:Y\to X$ be a fiber bundle and $k\in\mathbb Z^+$. The $k$-th order jet bundle of $\pi_{Y,X}$ is denoted by $\pi_{J^k Y,X}:J^k Y\to X$ and its elements by $j_x^k s$. Similarly, its fibers are denoted by $J_x^k Y$, for each $x\in X$. The $k$-th jet lift of a section $s\in\Gamma(\pi_{Y,X})$ is denoted by $j^k s\in\Gamma\left(\pi_{J^k Y,X}\right)$ and it is called holonomic section. Recall that, for $0\leq l<k$, the maps $\pi_{k,l}:J^k Y\to J^l Y$, $\pi_{k,l}(j^k_x s)=j^l_xs$, are fiber bundles, where we denote $J^0 Y=Y$. In particular, the first jet bundle of $\pi_{Y,X}$ may be regarded as
\begin{equation*}
J^1 Y=\left\{\gamma_y:T_x X\to T_y Y\mid\gamma_y\text{ linear},~x=\pi_{Y,X}(y),~(d\pi_{Y,X})_y\circ\gamma_y=\textrm{\normalfont id}_{T_x X}\right\}.
\end{equation*}

On the other hand, the $k$-th order jet bundle of a trivial bundle $\pi_{X\times F,X}:X\times F\to X$ is known as the $k$-th order jet of functions from $X$ to $F$. Since any section $s\in\Gamma(\pi_{X\times F,X})$ is of the form $s=\left(\textrm{\normalfont id}_X,f\right)$ for some function $f:X\to F$, elements of $J^k(X,F)=J^k(X\times F)$ are denoted by $j_x^k f$.

\subsection{Universal enveloping algebra}

We recall the main notions about the universal enveloping algebra of a Lie algebra. For an in-depth exposition see, for example, \cite{Va1984,Ha2015}. Let $G$ be a Lie group and $\mathfrak g$ be its Lie algebra. For each $g\in G$, we denote by $R_g:G\to G$ the right multiplication by $g$, by $\Ad_g:\mathfrak g\to\mathfrak g$ the adjoint representation of $G$, and by $\ad(\xi):\mathfrak g\to\mathfrak g$ the the adjoint representation of $\mathfrak g$, where $\xi\in\mathfrak g$.

\begin{theorem}
Let $\mathfrak g$ be a Lie algebra. There exists an associative algebra with identity, which is called \emph{universal enveloping algebra} of $\mathfrak g$ and denoted by $U(\mathfrak g)$, and a linear map $\iota:\mathfrak g\to U(\mathfrak g)$ such that
\begin{enumerate}
    \item $\iota([\xi,\eta])=\iota(\xi)\iota(\eta)-\iota(\eta)\iota(\xi)$ for each $\xi,\eta\in\mathfrak g$.
    \item $U(\mathfrak g)$ is generated, as a algebra, by $\{\iota(\xi)\mid\xi\in\mathfrak g\}$.
    \item If $\mathcal A$ is an associative algebra with identity and $i:\mathfrak g\to\mathcal A$ is linear and $i([\xi,\eta])=i(\xi)i(\eta)-i(\eta)i(\xi)$ for each $\xi,\eta\in\mathfrak g$, then there exists a unique algebra homomorphism $\phi: U(\mathfrak g)\to\mathcal A$ such that the following diagram is commutative,
    \begin{equation*}
    \begin{tikzpicture}
\matrix (m) [matrix of math nodes,row sep=3em,column sep=3em,minimum width=2em]
{	\mathfrak g & \mathcal A \\
	U(\mathfrak g) & \\};
\path[-stealth]
(m-1-1) edge [] node [left] {$\iota$} (m-2-1)
(m-1-1) edge [] node [above] {$i$} (m-1-2)
(m-2-1) edge [dashed] node [below] {$\phi$} (m-1-2);
\end{tikzpicture}
    \end{equation*}
\end{enumerate}
\end{theorem}

Universal enveloping algebras are useful in representation theory, but we do not develop this topic here.

\begin{theorem}[PBW Theorem]
Let $\{B_1,\dots,B_m\}$ be a basis of $\mathfrak g$, then
\begin{equation*}
\textrm{\normalfont span}\left\{\iota(\xi_1)^{\alpha_1}\dots\iota(\xi_m)^{\alpha_m}\mid\alpha_i\in\mathbb Z,~\alpha_i\geq 0,~1\leq i\leq m\right\},
\end{equation*}
where $\iota(\xi)^0=1$, the identity of $U(\mathfrak g)$, is a basis of $U(\mathfrak g)$ as a vector space. In particular, the map $\iota:\mathfrak g\to U(\mathfrak g)$ is injective.
\end{theorem}

Thanks to the above theorem, henceforth we identify $\mathfrak g\ni\xi\equiv\iota(\xi)\in U(\mathfrak g)$ and we regard $\mathfrak g\subset U(\mathfrak g)$ as a vector subspace.

\subsection{Partly ordered and anti-lexicographically ordered partitions}

Let $j\in\mathbb Z^+$. A partition of $\{1,\dots,j\}$ of length $l\in\{1,\dots,j\}$ is a tuple $\lambda=(\lambda_1,\dots,\lambda_l)$ where $\lambda_1,\dots,\lambda_l\subset\{1,\dots,j\}$ are disjoint subsets such that $\lambda_1\cup\dots\cup\lambda_l=\{1,\dots,j\}$. We denote by $\mathcal P(j)$ the family of all partitions of $\{1,\dots,j\}$. Likewise, we denote $j_i=|\lambda_i|$, the cardinality of $\lambda_i$, $1\leq i\leq l$. Of course, $j_1+\dots+j_l=j$ for each partition $\lambda=(\lambda_1,\dots,\lambda_l)\in\mathcal P(j)$.

\begin{definition}
A partition $\lambda=(\lambda_1,\dots,\lambda_l)\in\mathcal P(j)$ is \emph{partly ordered} if $\alpha_1^i<\dots<\alpha_{j_i}^i$ for each $1\leq i\leq l$, where we denote $\lambda_i=\{\alpha_1^i,\dots,\alpha_{j_i}^i\}$.
\end{definition}

Observe that we ask for the integers to be ordered only within each subset $\lambda_i$, $1\leq i\leq l$. The set of all partly ordered partitions of $\{1,\dots,j\}$ is denoted by $\mathcal P^+(j)$. At last, we consider the family of partly ordered partitions with $\alpha_1^1=1$, i.e.,
\begin{equation*}
\mathcal P_1^+(j)=\left\{\lambda=(\lambda_1,\dots,\lambda_l)\in\mathcal P^+(j)\mid 1\in\lambda_1\right\}.
\end{equation*}

\begin{proposition}
Let $j\in\mathbb Z^+$ and $1\leq l\leq j$. Fixed $j_1,\dots,j_l\in\{1,\dots,j\}$ such that $j_1+\dots+j_l=j$, there are exactly
\begin{equation*}
c(j_1,\dots,j_l)=\binom{j-1}{j_1-1}\prod_{i=2}^l\binom{j_i+\dots+j_l}{j_i}
\end{equation*}
different partitions $\lambda=(\lambda_1,\dots,\lambda_l)\in\mathcal P_1^+(j)$ with $|\lambda_i|=j_i$, $1\leq i\leq l$. In particular, the cardinality of $\mathcal P_1^+(j)$ is 
\begin{equation*}
\left|\mathcal P_1^+(j)\right|=\sum_{l=1}^j~\sum_{j_1+\dots+j_l=j}c(j_1,\dots,j_l).
\end{equation*}
\end{proposition}

\begin{proof}
For the first part, we write $\lambda_1=\{\alpha_1^1=1<\alpha_2^1<\dots<\alpha_{j_1}^1\}$ for each $\lambda=(\lambda_1,\dots,\lambda_l)\in\mathcal P_1^+(j)$. Thus, there are $\binom{j-1}{j_1-1}$ ways of choosing the elements $\alpha_2^1,\dots,\alpha_{j_1}^1$. We have $n-j_1$ remaining elements, whence there are $\binom{j-j_1}{j_2}$ choices for the elements of $\lambda_2$. Generally, there are $\binom{j-j_1\dots-j_{i-1}}{j_i}$ choices for the elements of $\lambda_i$, $2\leq i\leq l$. In short, the number of partitions in $\mathcal P_1^+(j)$ with $|\lambda_i|=j_i$, $1\leq i\leq l$, is exactly
\begin{equation*}
c(j_1,\dots,j_l)=\binom{j-1}{j_1-1}\prod_{i=2}^{l}\binom{j-j_1\dots-j_{i-1}}{j_i}.
\end{equation*}
We conclude by recalling that $j_1+\dots+j_l=j$. The second part is a straightforward consequence of the first one.
\end{proof}

Given a partition $\lambda=(\lambda_1,\dots,\lambda_l)\in\mathcal P(j)$, we consider two ways to derive a new partition of $\{1,\dots,j+1\}$. Namely, fixed $1\leq s\leq l$ we define
\begin{enumerate}
    \item $\lambda_{[s]}^+=\left(\lambda_1,\dots,\lambda_{s-1},\lambda_{s}\cup\{j+1\},\lambda_{s+1},\dots,\lambda_l\right)$.
    \item $\lambda_{[s]}^-=\left(\lambda_1,\dots,\lambda_{s},\{j+1\},\lambda_{s+1},\dots,\lambda_l\right)$.
\end{enumerate}
For instance, given $\lambda=\left(\{23\},\{1\}\right)\in\mathcal P(3)$, we have $\lambda_{[1]}^+=\left(\{234\},\{1\}\right)$ and $\lambda_{[1]}^-=\left(\{23\},\{4\},\{1\}\right)$. We have the following lemmas.

\begin{lemma}
Let $j\in\mathbb Z^+$, $\lambda=(\lambda_1,\dots,\lambda_l)\in\mathcal P_1^+(j)$ and $1\leq s\leq l$. Then $\lambda_{[s]}^+,\lambda_{[s]}^-\in\mathcal P_1^+(j+1)$.
\end{lemma}

\begin{lemma}
Let $j\in\mathbb Z^+$ and $\lambda=(\lambda_1,\dots,\lambda_{l})\in\mathcal P_1^+(j+1)$. Then there exists a unique $\hat\lambda\in\mathcal P_1^+(j)$ and $1\leq s\leq l$ such that $\lambda=\hat\lambda_{[s]}^+$ or $\lambda=\hat\lambda_{[s]}^-$.
\end{lemma}

\begin{proof}
Let $1\leq s\leq l$ be such that $j+1\in\lambda_{s}$. If $|\lambda_{s}|>1$, we pick
\begin{equation*}
\hat\lambda=(\lambda_1,\dots,\lambda_{s-1},\lambda_{s}-\{j+1\},\lambda_{s+1},\dots,\lambda_l)\in\mathcal P_1^+(j).
\end{equation*}
It is thus clear that $\hat\lambda_{[s]}^+=\lambda$. Analogously, if $|\lambda_{s}|=1$, then $s\geq 2$ and we choose
\begin{equation*}
\hat\lambda=(\lambda_1,\dots,\lambda_{s-1},\lambda_{s+1},\dots,\lambda_l)\in\mathcal P_1^+(j).
\end{equation*}
It is straightforward that $\hat\lambda_{[s-1]}^-=\lambda$.

Lastly, observe that if we have two different partitions $\hat\lambda,\tilde\lambda\in\mathcal P_1^+(j)$, then $\hat\lambda_{[s]}^+$, $\hat\lambda_{[s]}^-$, $\tilde\lambda_{[s]}^+$ and $\tilde\lambda_{[s]}^-$ are all distinct, whenever they are defined. Hence, the partition $\hat\lambda$ yielding $\lambda$ is unique.
\end{proof}

The previous lemmas have the following straightforward result.

\begin{proposition}\label{prop:disjointunionpartly}
Let $j\in\mathbb Z^+$ and denote by $\sqcup$ the disjoint union. Then
\begin{equation*}
\mathcal P_1^+(j+1)=\bigsqcup_{\lambda\in\mathcal P_1^+(j)}~\bigsqcup_{s=1}^{l}\left(\left\{\lambda_{[s]}^+\right\}\sqcup\left\{\lambda_{[s]}^-\right\}\right).
\end{equation*}
\end{proposition}

Additionally, we are interested in anti-lexicographically ordered partitions.

\begin{definition}
A partition $\lambda=\left(\lambda_1,\dots,\lambda_l\right)\in\mathcal P(j)$ is \emph{anti-lexicographically ordered} if $\max\lambda_1<\dots<\max\lambda_l$, that is to say, if each subset contains the highest available number when going from right to left.
\end{definition}

We denote by $\mathcal P^a(j)$ the set of all anti-lexicographically ordered partitions. In order to illustrate both definitions, consider $\mathcal P(3)$. Then $\lambda=\left(\{12\},\{3\}\right)$ is both partly and anti-lexicographically ordered, $\lambda=\left(\{3\},\{12\}\right)$ is partly ordered but not anti-lexicographically ordered, $\lambda=\left(\{21\},\{3\}\right)$ is anti-lexicographically ordered, but not partly ordered, and $\lambda=\left(\{3\},\{21\}\right)$ is not partly nor anti-lexicographically ordered. 

\begin{proposition}[\cite{Vi2013}]\label{prop:amountapartitions}
Let $j\in\mathbb Z^+$ and $1\leq l\leq j$. Fixed $j_1,\dots,j_l\in\{1,\dots,j\}$ such that $j_1+\dots+j_l=j$, there are exactly
\begin{equation*}
N(j_1,\dots,j_l)=\prod_{i=2}^l\binom{j_1+\dots+j_i-1}{j_i-1}
\end{equation*}
different partitions $\lambda=\left(\lambda_1,\dots,\lambda_l\right)\in\mathcal P^a(j)$ with $|\lambda_i|=j_i$, $1\leq i\leq l$.
\end{proposition}

Given a partition $\lambda=\left(\lambda_1,\dots,\lambda_l\right)\in\mathcal P(j)$, we derive a partition of $\{1,\dots,j+1\}$ as follows. Fix $1\leq s\leq l$ and define
\begin{equation*}
\lambda_{[s]}=\left(\hat\lambda_1,\dots,\hat\lambda_{s}\cup\{1\},\dots\hat\lambda_l\right),
\end{equation*}
where $\hat\lambda_i=\{\alpha_1^i+1,\dots,\alpha_{j_i}^i+1\}$ for each $1\leq i\leq l$. Similarly, we define
\begin{equation*}
\lambda_{[0]}=\left(\{1\},\hat\lambda_1,\dots,\hat\lambda_l\right).
\end{equation*}

\begin{lemma}
Let $j\in\mathbb Z^+$, $\lambda=(\lambda_1,\dots,\lambda_l)\in\mathcal P^a(j)$ and $0\leq s\leq l$. Then $\lambda_{[s]}\in\mathcal P^a(j+1)$.
\end{lemma}

\begin{lemma}[\cite{Vi2013}]
Let $j\in\mathbb Z^+$ and $\lambda=(\lambda_1,\dots,\lambda_l)\in\mathcal P^a(j+1)$. Then there exists a unique $\hat\lambda\in\mathcal P^a(j)$ and $0\leq s\leq l$ such that $\lambda=\hat\lambda_{[s]}$.
\end{lemma}

\begin{proposition}\label{prop:disjointunionantilexicographically}
For each $j\in\mathbb Z^+$ we have
\begin{equation*}
\mathcal P^a(j+1)=\bigsqcup_{\lambda\in\mathcal P^a(j)}~\bigsqcup_{s=0}^l\left\{\lambda_{[s]}\right\}.
\end{equation*}
\end{proposition}

\section{Trivialization for the higher order jet bundle}\label{sec:trivalization}

Let $G$ be a Lie group bundle, $\mathfrak g$ be its Lie algebra and $X$ be a smooth manifold. In the following, we regard the jet defined by a (local) function $g:X\to G$ as an element of the vector bundle $j_x^1 g\equiv (dg)_x\in T^*X\otimes TG$ over $X$. Let $\nabla^X:\Gamma(TX)\to\Gamma(T^*X\otimes TX)$ be a linear connection on the tangent bundle of $X$, and consider the dual connection, $\nabla:\Gamma(T^*X)\to\Gamma(T^*X\otimes T^*X)$, on the cotangent bundle. For each $k\in\mathbb Z^+$, let 
\begin{equation*}
\nabla^{(k)}:\Gamma\left(\bigotimes^k T^*X\to X\right)\longrightarrow\Gamma\left(\bigotimes^{k+1}T^*X\to X\right)    
\end{equation*}
be the corresponding tensor product connection on $\bigotimes^k T^*X\to X$, which is again a linear connection \cite[\S 2.4]{MaSa2000}. Note that $\nabla^{(1)}=\nabla$. At last, we extend our linear connections to maps,
\begin{equation*}
\tilde\nabla^{(k)}:\Gamma\left(\bigotimes^k T^*X\otimes\mathfrak g\to X\right)\longrightarrow\Gamma\left(\bigotimes^{k+1}T^*X\otimes\mathfrak g\to X\right),
\end{equation*}
Observe that this works although $\nabla^{(k)}$ is not tensorial, since $\mathfrak g$ is not a vector bundle, but a vector space. Indeed, if $(U,x)$ is a coordinated chart of $X$, where we write $x=(x^\mu)=(x^1,\dots,x^n)$, and $\{B_1,\dots,B_m\}$ is a basis of $\mathfrak g$, we set
\begin{equation*}
\tilde\nabla^{(k)}\left(\xi_{\mu_1,\dots,\mu_k}^\alpha\,dx^{\mu_1}\otimes\dots\otimes dx^{\mu_k}\otimes B_\alpha\right)=\nabla^{(k)}\left(\xi_{\mu_1,\dots,\mu_k}^\alpha\,dx^{\mu_1}\otimes\dots\otimes dx^{\mu_k}\right)\otimes B_\alpha,    
\end{equation*}
for each $\xi_{\mu_1,\dots,\mu_k}^\alpha:X\to\mathbb R$, $1\leq\mu_1,\dots,\mu_k\leq n$, $1\leq\alpha\leq m$.

\begin{theorem}
For each $k\in\mathbb Z^+$, the following is an injective morphism of bundles over $X\times G$,
\begin{equation}\label{eq:injectionhigherorderjet}
J^k(X,G)\hookrightarrow G\times\bigoplus_{j=1}^k\left(\bigotimes^j T^*X\otimes\mathfrak g\right),\qquad j_x^k g\mapsto\left(g(x),\xi^{(1)}(x),\dots,\xi^{(k)}(x)\right),
\end{equation}
where $$\xi^{(1)}(x)=\left(dR_{g(x)^{-1}}\right)_{g(x)}\circ(dg)_x\in T_x^*X\otimes\mathfrak g,$$ and $$\xi^{(j)}(x)=(\tilde\nabla^{(j-1)}\xi^{(j-1)})(x)\in\bigotimes^j T_x^*X\otimes\mathfrak g, \quad 2\leq j\leq k.$$
\end{theorem}

\begin{proof}
We show the result by induction in $k$. For $k=1$, the first jet defined by a (local) function $g:X\to G$ is given by $j_x^1g=(dg)_x: T_x X\to T_{g(x)} G$. Subsequently, the map is given by composing this jet with the right trivialization $T_g G\ni U_g\mapsto(g,(dR_{g^{-1}})_g(U_g))\in G\times\mathfrak g$. Namely,
\begin{equation*}
J^1(X,G)\hookrightarrow G\times(T^*X\otimes\mathfrak g),\qquad j_x^1g\mapsto\big(g(x),(dR_{g(x)^{-1}})\circ(dg)_x\big),
\end{equation*}
which is injective since it is an isomorphism of fiber bundles.

Now, given $k>1$, we assume that \eqref{eq:injectionhigherorderjet} is injective for $k-1$. Let $j_x^k g,j_x^k h\in J^k(X,G)$ be such that
\begin{equation*}
\left(g(x),\xi^{(1)}(x),\dots,\xi^{(k)}(x)\right)\neq \left(h(x),\eta^{(1)}(x),\dots,\eta^{(k)}(x)\right),
\end{equation*}
where $\xi^{(1)}=\left(dR_{g^{-1}}\right)_g\circ dg$, $\eta^{(1)}=\left(dR_{h^{-1}}\right)_h\circ dh$, $\xi^{(j)}=\tilde\nabla^{(j-1)}\xi^{(j-1)}$ and $\eta^{(j)}=\tilde\nabla^{(j-1)}\eta^{(j-1)}$, $2\leq j\leq k$. If either $g(x)\neq h(x)$ or $\xi^{(j)}(x)\neq\eta^{(j)}(x)$ for some $1\leq j\leq k-1$, then $j_x^k g\neq j_x^k h$ by the induction hypothesis. At last, suppose that $g(x)=h(x)$, $\xi^{(j)}(x)=\eta^{(j)}(x)$, $1\leq j\leq k-1$ and $\xi^{(k)}(x)\neq\eta^{(k)}(x)$. Recall that, in coordinates, we may write (cf. \cite[\S 2.4]{MaSa2000}) $\nabla^{(k-1)}=\textrm{\normalfont d}+\Gamma^{(k-1)}$, where $\textrm{\normalfont d}:\Gamma\left(\bigotimes^{k-1}T^*X\right)\to\Gamma\left(\bigotimes^k T^*X\right)$ is the exterior derivative (which is well-defined, since we are working on a coordinated chart) and $\Gamma^{(k-1)}\in\Gamma\left(T^*X\otimes\textrm{\normalfont End}\left(\bigotimes^{k-1}T^*X\right)\right)$, where $\textrm{\normalfont End}\left(\bigotimes^{k-1}T^*X\right)$ is the bundle of endomorphisms of $\bigotimes^{k-1}T^*X$. We have $\Gamma^{(k-1)}\left(\xi^{(k-1)}(x)\right)=\Gamma^{(k-1)}\left(\eta^{(k-1)}(x)\right)$ and $\tilde\nabla^{(k-1)}\xi^{(k-1)}(x)\neq\tilde\nabla^{(k-1)}\eta^{(k-1)}(x)$. Thence $\left(\textrm{\normalfont d}\xi^{(k-1)}\right)_x\neq\left(\textrm{\normalfont d}\eta^{(k-1)}\right)_x$. Subsequently, $j_x^k g\neq j_x^k h$, and we conclude.
\end{proof}

By means of this map, we denote elements of the $k$-th order jet bundle by
\begin{equation*}
\left(g,\xi^{(1)},\dots,\xi^{(k)}\right)\in J^k(X,G),
\end{equation*}
for some $g\in G$ and $\xi^{(j)}\in\bigotimes^j T^*X\otimes\mathfrak g$, $1\leq j\leq k$, all of them projecting to certain $x\in X$. 

\begin{remark}
Except for $k=1$, the morphism \eqref{eq:injectionhigherorderjet} need not be surjective, that is, there are elements in $G\times\bigoplus_{j=1}^k\left(\bigotimes^j T^*X\otimes\mathfrak g\right)$, $k\geq 2$, that does not represent an element of $J^k(X,G)$. However, Example \ref{example} below will show that \eqref{eq:injectionhigherorderjet} is always bijective for the particular case $X=\mathbb R$.
\end{remark}

\section{Lie groupoid structure of the higher order jet bundle} \label{sec:structure}

Since $G$ is a Lie group, for each $k\in\mathbb Z^+$ the jet bundle $J^k(X,G)\to X$ is endowed with a natural fibered multiplication. More specifically, the following map is well-defined:
\begin{equation*}
J^k(X,G)\times_X J^k(X,G)\to J^k (X,G),\qquad \left(j_x^k g,j_x^k h\right)\mapsto j_x^k(gh),
\end{equation*}
where $\times_X$ denotes the fibered product of bundles over $X$. Therefore, $J^k(X,G)\rightrightarrows X$ is a Lie groupoid with both the source and the target maps being $\pi_{J^k(X,G),X}$. The aim of this section is to compute the expression of the fibered multiplication under the identification \eqref{eq:injectionhigherorderjet}. We start with the following technical lemma.

\begin{lemma}\label{lemma:nabladjoint}
Let $j\in\mathbb Z^+$, $g:X\to G$ and $\chi^{(j)}\in\Gamma\left(\bigotimes^j T^*X\otimes\mathfrak g\to X\right)$, and denote $\chi=\left(dR_{g^{-1}}\right)_g\circ dg\in\Gamma\left(T^*X\otimes\mathfrak g\right)$. Then
\begin{equation*}
\tilde\nabla^{(j)}\left(\Ad_g\circ\chi^{(j)}\right)=\ad(\chi)\left(\Ad_g\circ\chi^{(j)}\right)+\Ad_g\circ\tilde\nabla^{(j)}\chi^{(j)}.
\end{equation*}
\end{lemma}

\begin{proof}
Let $U,U_1,\dots,U_j\in\mathfrak X(X)$, and recall that we denote by $\nabla^X$ the linear connection on the tangent bundle, $TX$. We pick a curve $\gamma:(-\epsilon,\epsilon)\to X$ such that $\gamma(0)=x$ and $\gamma'(0)=U(x)$. For the sake of brevity, we denote by $\varphi:X\to\mathfrak g$ the map defined as $x\mapsto\chi^{(j)}(x)(U_1(x),\dots,U_j(x))$. Therefore,
\begin{multline*}
U\left(\Ad_g(\varphi)\right)(x)=\left.\frac{d}{dt}\right|_{t=0}\Ad_{(g\circ\gamma)}\left(\varphi\circ\gamma\right)(t)=\left.\frac{d}{dt}\right|_{t=0}\Ad_{(g\circ\gamma)g(x)^{-1}}\left(\Ad_{g(x)}\left(\varphi\circ\gamma\right)\right)(t)\\
=\ad\left(\chi(x)\left(U(x)\right)\right)\left(\Ad_{g(x)}\left(\varphi(x)\right)\right)+\Ad_{g(x)}\left((d\varphi)_x(U(x))\right)\\
=\ad\left(\chi(x)\left(U(x)\right)\right)\left(\Ad_{g(x)}\left(\varphi(x)\right)\right)+\Ad_{g(x)}\left(U(x)\left(\varphi(x)\right)\right).
\end{multline*}
As a result, we get
\begin{equation*}
U\left(\Ad_g(\varphi)\right)=\ad\left(\chi(U)\right)\left(\Ad_g(\varphi)\right)+\Ad_g(U(\varphi)).
\end{equation*}
By using this, we conclude
\begin{multline*}
\tilde\nabla_U^{(j)}\left(\Ad_g\circ\chi^{(j)}\right)(U_1,\dots,U_j)\\
=U\left(\left(\Ad_g\circ\chi^{(j)}\right)(U_1,\dots,U_j)\right)-\sum_{i=1}^j\left(\Ad_g\circ\chi^{(j)}\right)\left(U_1,\dots,\nabla_U^X U_i,\dots,U_j\right)\\
=\ad\left(\chi(U)\right)\left(\Ad_g(\varphi)\right)+\Ad_g(U(\varphi))-\Ad_g\left(\sum_{i=1}^j\chi^{(j)}\left(U_1,\dots,\nabla_U^X U_i,\dots,U_j\right)\right)\\
=\ad\left(\chi(U)\right)\left(\Ad_g(\varphi)\right)+\Ad_g\left(U(\varphi)-\sum_{i=1}^j\chi^{(j)}\left(U_1,\dots,\nabla_U^X U_i,\dots,U_j\right)\right)\\
=\ad(\chi(U))\left(\left(\Ad_g\circ\chi^{(j)}\right)(U_1,\dots,U_j)\right)+\Ad_g\left(\tilde\nabla_U^{(j)}\chi^{(j)}(U_1,\dots,U_j)\right).
\end{multline*}
\end{proof}

Recall that for each $k,r\in\mathbb Z^+$ there is an injective map $J^{k+r}(X,G)\hookrightarrow J^k\left(J^r(X,G)\right), j_x^{k+r}g\mapsto j_x^k\left(j^r g\right)$. This allows us to identify
\begin{equation}\label{eq:injection}
J^k(X,G)\ni j_x^k g\simeq j_x^1(j^{k-1}g)\in J^1\left(J^{k-1}(X,G)\right).
\end{equation}

\begin{theorem}
Let $x\in X$ and $\left(g,\xi^{(1)},\dots,\xi^{(k)}\right),\left(h,\eta^{(1)},\dots,\eta^{(k)}\right)\in J_x^k(X,G)$. Then their fibered product is given by $\left(g,\xi^{(1)},\dots,\xi^{(k)}\right)\left(h,\eta^{(1)},\dots,\eta^{(k)}\right)=\left(gh,\zeta^{(1)},\dots,\zeta^{(k)}\right)$, where
\begin{equation}\label{eq:multiplicationelements}
\zeta^{(j)}=\xi^{(j)}+\sum_{\lambda\in\mathcal P^a(j)}\ad\left(\xi^{({j_{l-1}})}\right)\dots\ad\left(\xi^{({j_1})}\right)\left(\Ad_g\circ\eta^{(j_l)}\right),\qquad 1\leq j\leq k.
\end{equation}
\end{theorem}

\begin{proof}
We prove it by induction in $k$. For $k=1$, consider two (local) functions $\tilde g,\tilde h:X\to G$ such that $\tilde g(x)=g$, $\tilde h(x)=h$, $\left(dR_{g^{-1}}\right)_g\circ(d\tilde g)_x=\xi^{(1)}$ and $\left(dR_{h^{-1}}\right)_h\circ(d\tilde h)_x=\eta^{(1)}$. Recall that $\left(g,\xi^{(1)}\right)\left(h,\eta^{(1)}\right)\simeq j_x^1(\tilde g\tilde h)$. By using this, \eqref{eq:injectionhigherorderjet} and the Leibniz rule, it can be checked that
\begin{equation*}
\left(g,\xi^{(1)}\right)\left(h,\eta^{(1)}\right)=\left(gh,\left(dR_{(gh)^{-1}}\right)_{gh}\circ d\left(\tilde g\tilde h\right)_x\right)=\left(gh,\xi^{(1)}+\Ad_g\circ\eta^{(1)}\right)=\left(gh,\zeta^{(1)}\right).
\end{equation*}

Now, given $k>1$, we assume that the result holds for $k-1$ and we pick a (local) function $\tilde g:X\to G$ such that $j_x^k\tilde g\simeq\left(g,\xi^{(1)},\dots,\xi^{(k)}\right)$ via \eqref{eq:injectionhigherorderjet}. Moreover, we have $j^{k-1}\tilde g\simeq\left(\tilde g,\tilde\xi^{(1)},\dots,\tilde\xi^{(k-1)}\right)$. In the same vein, we pick a (local) function $\tilde h:X\to G$ such that $j_x^k\tilde h\simeq\left(h,\eta^{(1)},\dots,\eta^{(k)}\right)$. By the induction hypothesis, $\left(j^{k-1}\tilde g\right)\left(j^{k-1}\tilde h\right)\simeq\left(\tilde g\tilde h,\tilde\zeta^{(1)},\dots,\tilde\zeta^{(k-1)}\right)$. Thanks to this and \eqref{eq:injection}, we may write
\begin{equation*}
\left(g,\xi^{(1)},\dots,\xi^{(k)}\right)\left(h,\eta^{(1)},\dots,\eta^{(k)}\right)\simeq j_x^1\left(\left(j^{k-1}\tilde g\right)\left(j^{k-1}\tilde h\right)\right)\simeq\left(gh,\zeta^{(1)},\dots,\zeta^{(k-1)},\tilde\nabla^{(k-1)}\tilde\zeta^{(k-1)}(x)\right)
\end{equation*}
By recalling that $\tilde\nabla^{(j)}\tilde\xi^{(j)}(x)=\xi^{(j+1)}$, $1\leq j\leq k-1$, and analogous for $\tilde\eta^{(j)}$, and by using Proposition \ref{prop:disjointunionantilexicographically}, we finish
\begin{multline*}
\tilde\nabla^{(k-1)}\tilde\zeta^{(k-1)}(x)=\\
\tilde\nabla^{(k-1)}\left(\tilde\xi^{(k-1)}+\sum_{\lambda\in\mathcal P^a(k-1)}\ad\left(\tilde\xi^{({j_{l-1}})}\right)\dots\ad\left(\tilde\xi^{({j_1})}\right)\left(\Ad_{\tilde g}\circ\tilde\eta^{(j_l)}\right)\right)(x)\\
=\xi^{(k)}+\sum_{\lambda\in\mathcal P^a(k-1)}\sum_{s=1}^{l-1}\ad\left(\xi^{({j_{l-1}})}\right)\dots\ad\left(\xi^{(j_s+1)}\right)\dots\ad\left(\xi^{({j_1})}\right)\left(\Ad_g\circ\eta^{(j_l)}\right)\\
+\sum_{\lambda\in\mathcal P^a(k-1)}\ad\left(\xi^{({j_{l-1}})}\right)\dots\ad\left(\xi^{({j_1})}\right)\left(\ad\left(\xi^{(1)}\right)\left(\Ad_g\circ\eta^{(j_l)}\right)+\Ad_g\circ\eta^{(j_l+1)}\right)\\
=\xi^{(k)}+\sum_{\lambda\in\mathcal P^a(k)}\ad\left(\xi^{({j_{l-1}})}\right)\dots\ad\left(\xi^{({j_1})}\right)\left(\Ad_g\circ\eta^{(j_l)}\right)=\zeta^{(k)}.
\end{multline*}
For the second equation, we have used that for each $j\in\mathbb Z^+$ and $j_1,j_2\in\mathbb Z^+$ such that $j_1+j_2=j$ we have 
\begin{multline}\label{eq:leibnizrule}
\tilde\nabla^{(j)}\ad\left(\chi^{(j_1)}\right)\left(\Ad_g\circ\chi^{(j_2)}\right)=\\
\ad\left(\tilde\nabla^{(j_1)}\chi^{(j_1)}\right)\left(\Ad_g\circ\chi^{(j_2)}\right)+\ad\left(\chi^{(j_1)}\right)\left(\ad\left(\chi\right)\left(\Ad_g\circ\chi^{(j_2)}\right)+\Ad_g\circ\tilde\nabla^{(j_2)}\chi^{(j_2)}\right)
\end{multline}
for each (local) section $\chi^{(j_i)}\in\Gamma\left(\bigotimes^{j_i}T^*X\otimes\mathfrak g\to X\right)$, $i=1,2$, and each (local) function $g:X\to G$, where $\chi=\left(dR_{g^{-1}}\right)_g\circ dg\in\Gamma\left(T^*X\otimes\mathfrak g\to X\right)$. This is a straightforward consequence of the Leibniz rule, Lemma \ref{lemma:nabladjoint} and the associativity of the tensor product of linear connections, i.e., $\tilde\nabla^{(j)}=\tilde\nabla^{(j_1)}\otimes\tilde\nabla^{(j_2)}$.
\end{proof}

Observe that the identity element is given by $\left(e,0,\dots,0\right)\in J^k(X,G)$, where $e\in G$ is the identity element.

\begin{corollary}
Let $\left(g,\xi^{(1)},\dots,\xi^{(k)}\right)\in J^k(X,G)$. Then its inverse is given by $\left(g,\xi^{(1)},\dots,\xi^{(k)}\right)^{-1}=\left(g^{-1},\omega^{(1)},\dots,\omega^{(k)}\right)$, where
\begin{equation}\label{eq:inverselement}
\omega^{(j)}=\sum_{\lambda\in\mathcal P^a(j)}(-1)^l~\Ad_{g^{-1}}\circ\left(\ad\left(\xi^{(j_1)}\right)\dots\ad\left(\xi^{(j_{l-1})}\right)\left(\xi^{(j_l)}\right)\right),\qquad 1\leq j\leq k.
\end{equation}
\end{corollary}

\begin{proof}
For $k=1$, given $\left(g,\xi^{(1)}\right)\in J^1(X,G)$, it is easy to check using \eqref{eq:multiplicationelements} that
\begin{equation*}
\left(g,\xi^{(1)}\right)\left(g^{-1},-\Ad_{g^{-1}}\circ\xi^{(1)}\right)=\left(e,0\right),
\end{equation*}
whence $\left(g,\xi^{(1)}\right)^{-1}=\left(g^{-1},-\Ad_{g^{-1}}\circ\xi^{(1)}\right)$.

Now, given $k>1$, we assume that the result holds for $k-1$. Let $\tilde g:X\to G$ be a (local) function such that $j_x^k\tilde g\simeq\left(g,\xi^{(1)},\dots,\xi^{(k)}\right)$ via \eqref{eq:injectionhigherorderjet}. Thanks to \eqref{eq:injection} and the induction hypothersis, we may write 
\begin{equation*}
\left(g,\xi^{(1)},\dots,\xi^{(k)}\right)^{-1}\simeq\left(j_x^k\tilde g\right)^{-1}\simeq j_x^1\left(j^{k-1}\tilde g\right)^{-1}=j_x^1\left(\tilde g^{-1},\tilde\omega^{(1)},\dots,\tilde\omega^{(k-1)}\right).
\end{equation*} Hence,
\begin{equation*}
\left(g,\xi^{(1)},\dots,\xi^{(k)}\right)^{-1}=\left(g^{-1},\omega^{(1)},\dots,\omega^{(k-1)},\left(\tilde\nabla^{(k-1)}\tilde\omega^{(k-1)}\right)(x)\right).
\end{equation*}
By recalling that $\left(\tilde\nabla^{(j)}\tilde\xi^{(j)}\right)(x)=\xi^{(j+1)}$, $1\leq j\leq k-1$, and by using Proposition \ref{prop:disjointunionantilexicographically}, we conclude:
\begin{multline*}
\left(\tilde\nabla^{(k-1)}\tilde\omega^{(k-1)}\right)(x)=\tilde\nabla^{(k-1)}\left(\sum_{\lambda\in\mathcal P^a(k-1)}(-1)^l~\Ad_{\tilde g^{-1}}\circ\left(\ad\left(\tilde\xi^{(j_1)}\right)\dots\ad\left(\tilde\xi^{(j_{l-1})}\right)\left(\tilde\xi^{(j_l)}\right)\right)\right)(x)\\
=\sum_{\lambda\in\mathcal P^a(k-1)}(-1)^{l+1}~\ad\left(\Ad_{g^{-1}}\circ\xi^{(1)}\right)\left(\Ad_{g^{-1}}\circ\left(\ad\left(\xi^{(j_1)}\right)\dots\ad\left(\xi^{(j_{l-1})}\right)\left(\xi^{(j_l)}\right)\right)\right)\\
+\sum_{\lambda\in\mathcal P^a(k-1)}(-1)^l~\Ad_{g^{-1}}\circ\sum_{s=1}^{l-1}\left(\ad\left(\xi^{(j_1)}\right)\dots\ad\left(\xi^{(j_s+1)}\right)\dots\ad\left(\xi^{(j_{l-1})}\right)\left(\xi^{(j_l)}\right)\right)\\
+\sum_{\lambda\in\mathcal P^a(k-1)}(-1)^l~\Ad_{g^{-1}}\circ\left(\ad\left(\xi^{(j_1)}\right)\dots\ad\left(\xi^{(j_{l-1})}\right)\left(\xi^{(j_l+1)}\right)\right)\\
=\sum_{\lambda\in\mathcal P^a(k)}(-1)^l\Ad_{g^{-1}}\circ\left(\ad\left(\xi^{(j_1)}\right)\dots\ad\left(\xi^{(j_{l-1})}\right)\left(\xi^{(j_l)}\right)\right)=\omega^{(k)},
\end{multline*}
where we have used \eqref{eq:leibnizrule} and the fact that $\ad\left(\Ad_h(\xi)\right)\left(\Ad_h(\eta)\right)=\Ad_h\left(\ad(\xi)(\eta)\right)$ for each $h\in G$ and $\xi,\eta\in\mathfrak g$.
\end{proof}

Thanks to Proposition \ref{prop:amountapartitions}, we may rewrite \eqref{eq:multiplicationelements} as 
\begin{equation*}
\zeta^{(j)}=\xi^{(j)}+\sum_{l=1}^j~\sum_{j_1+\dots+j_l=j}N(j_1,\dots,j_l)~\ad\left(\xi^{({j_{l-1}})}\right)\dots\ad\left(\xi^{({j_1})}\right)\left(\Ad_g\circ\eta^{(j_l)}\right),\qquad 1\leq j\leq k.
\end{equation*}
Similarly, \eqref{eq:inverselement} may be expressed as
\begin{equation*}
\omega^{(j)}=\sum_{l=1}^j~\sum_{j_1+\dots+j_l=j}~(-1)^l~N(j_1,\dots,j_l)~\Ad_{g^{-1}}\circ\left(\ad\left(\xi^{(j_1)}\right)\dots\ad\left(\xi^{(j_{l-1})}\right)\left(\xi^{(j_l)}\right)\right),\qquad 1\leq j\leq k.
\end{equation*}

\begin{example}[Second order jet]
Let us compute the explicit expressions for the case $k=2$. Let $\left(g,\xi^{(1)},\xi^{(2)}\right),\left(h,\eta^{(1)},\eta^{(2)}\right)\in J_x^2(X,G)$ and denote by $[\cdot,\cdot]$ the Lie bracket on $\mathfrak g$. An easy check from \eqref{eq:multiplicationelements} yields
\begin{equation*}
\left(g,\xi^{(1)},\xi^{(2)}\right)\left(h,\eta^{(1)},\eta^{(2)}\right)=\left(gh,\xi^{(1)}+\Ad_g\circ\eta^{(1)},\xi^{(2)}+\Ad_g\circ\eta^{(2)}+\left[\xi^{(1)},Ad_g\circ\eta^{(1)}\right]\right).
\end{equation*}
In the same vein, from \eqref{eq:inverselement} we obtain
\begin{equation*}
\left(g,\xi^{(1)},\xi^{(2)}\right)^{-1}=\left(g^{-1},-\,Ad_{g^{-1}}\circ\xi^{(1)},-\Ad_{g^{-1}}\circ\xi^{(2)}+\Ad_{g^{-1}}\circ\left[\xi^{(1)},\xi^{(1)}\right]\right).
\end{equation*}
\end{example}

\section{Local expression of the trivialization}\label{sec:coordinated}

In this section, we work in a coordinated chart $(U,x)$ of $X$, where we write $x=(x^\mu)=(x^1,\dots,x^n)$. This enables us to (locally) identify $X=\mathbb R^n$ and, hence, $\bigotimes^k T^*X=X\times\bigotimes^k\mathbb R^n$. Note that sections of that bundle are of the form $\xi^{(k)}=\left(\textrm{\normalfont id}_X,\hat\xi^{(k)}\right)$ for some function $\hat\xi^{(k)}:X\to\bigotimes^k\mathbb R^n$. By abusing the notation, we identify $\xi^{(k)}=\hat\xi^{(k)}$. This way, if we pick the canonical flat connection on the trivial bundle $T^*X=X\times\mathbb R^n\to X$, then we have $\nabla^{(k)}\alpha=d\alpha$ for each $\alpha\in\Gamma\left(X\times\bigotimes^k\mathbb R^n\to X\right)$.

\begin{example}\label{example}
Let $X=\mathbb R$ and denote by $t\in\mathbb R$ its global coordinate. In this case, a function $g:\mathbb R\to G$ is a just curve on $G$ and, thus, its derivative, $dg:T\mathbb R\to TG$, is uniquely determined by its velocity, $\dot g:\mathbb R\to g^*(TG)$. Namely, for each $t\in\mathbb R$ we have
\begin{equation*}
(dg)_t:T_t\mathbb R\simeq\mathbb R\to T_{g(t)}G,\qquad v\mapsto(dg)_t(v)=\dot g(t)v.
\end{equation*}
By using the velocity instead of the derivative, we get a map analogous to \eqref{eq:injectionhigherorderjet},
\begin{equation*}
J^k(\mathbb R,G)\to G\times k\mathfrak g,\qquad j_t^k g\mapsto\left(g(t),\xi^{(1)}(t),\dots,\xi^{(k)}(t)\right),
\end{equation*}
where $k\mathfrak g=\mathfrak g\oplus\overset{\underset{\smile}{k}}{\dots}\oplus\mathfrak g$ and
\[\xi^{(1)}(t)=\left(dR_{g(t)^{-1}}\right)_{g(t)}\left(\dot g(t)\right)\in\mathfrak g,
\]
\[
\xi^{(j)}(t)=\dot\xi^{(j-1)}(t)\in\mathfrak g, \quad 2\leq j\leq k.
\]
Note that the information about the base point $t\in\mathbb R$ is lost and, subsequently, the new map is not injective anymore. Nevertheless, if we restrict ourselves to a single fiber ---for instance, the fiber over $t=0$--- we obtain an isomorphism of bundles over $G$: $J_0^k(\mathbb R,G)\simeq G\times k\mathfrak g$. The group structure of $J_0^k(\mathbb R,G)$ under this identification was originally investigated in \cite{Vi2013}.
\end{example}

Additionally, we now assume that we are working with a matrix group, i.e. $G\subset GL(N,\mathbb R)$ and $N\in\mathbb Z^+$. Hence, $(dR_h)_g(U_g)=U_gh$, is the matrix multiplication for each $g,h\in G$ and $U_g\in T_g G$. Let $g=(g_{ab})_{a,b=1}^N:X\to G$ be a (local) function, where $g_{ab}:X\to\mathbb R$, $1\leq a,b\leq N$. Then, for each multi-index $J$, we may consider the partial derivative of $g$ as the following (local) function 
\begin{equation*}
\frac{\partial^{|J|} g}{\partial x^J}=\left(\frac{\partial^{|J|} g_{ab}}{\partial x^J}\right)_{a,b=1}^N:X\to\mathfrak{gl}(N,\mathbb R).
\end{equation*}
Subsequently, the tangent map of $g$ is given by $dg=(\partial g/\partial x^\mu) dx^\mu$, and analogous for higher order tangent maps: $d^2 g\equiv d(dg)=(\partial^2 g/\partial x^\mu\partial x^\nu)dx^\nu\otimes dx^\mu$, and so on. In order to keep the notation simple, it will be useful to define the maps
\begin{equation*}
\xi_{\mu_1\dots\mu_j}=\frac{\partial^j g}{\partial x^{\mu_1}\dots\partial x^{\mu_j}}g^{-1}:X\to \mathfrak{gl}(N,\mathbb R),\qquad1\leq\mu_1,\dots,\mu_j\leq n,\quad j\in\mathbb Z^+.
\end{equation*}

Observe that $\xi_{\mu_{\sigma_1(1)}\dots\mu_{\sigma_j(j)}}=\xi_{\mu_1\dots\mu_j}$ for any permutation $\sigma=(\sigma_1,\dots,\sigma_j)\in\mathcal S_j$. Hence, we can always rearrange the indices to be in ascending order. By the Leibniz's rule, it is easily checked that
\begin{equation}\label{eq:partialderivativesxi}
\frac{\partial\xi_{\mu_1\dots\mu_j}}{\partial x^\nu}=\xi_{\mu_1\dots\mu_j\nu}-\xi_{\mu_1\dots\mu_j}\xi_\nu,\qquad 1\leq\mu_1,\dots,\mu_j,\nu\leq n,\quad j\in\mathbb Z^+.
\end{equation}
Lastly, given $1\leq\mu_1,\dots,\mu_j\leq n$ and a subset $\kappa=\{\kappa_1,\dots,\kappa_r\}\subset\{1,\dots,j\}$, $1\leq r\leq j$, we denote
\begin{equation*}
\xi_{\kappa(\mu_1\dots\mu_j)}=\xi_{\mu_{\kappa_1}\dots\mu_{\kappa_r}}.
\end{equation*}
This way, given a partition $\lambda=(\lambda_1,\dots,\lambda_l)\in\mathcal P(j)$ we denote
\begin{equation*}
\xi_{\lambda(\mu_1\dots\mu_j)}=\prod_{i=1}^l\xi_{\lambda_i(\mu_1\dots\mu_j)}=\xi_{\lambda_1(\mu_1\dots\mu_j)}\dots\xi_{\lambda_l(\mu_1\dots\mu_j)}.
\end{equation*}

\begin{lemma}\label{lemma:partialxi}
For each $k\in\mathbb Z^+$, $k\geq 2$, $1\leq\mu_1,\dots,\mu_k\leq n$ and $\lambda\in\mathcal P_1^+(k-1)$ we have
\begin{equation*}
\frac{\partial}{\partial x^{\mu_k}}\xi_{\lambda(\mu_1\dots\mu_{k-1})}=\sum_{s=1}^l\left(\xi_{\lambda_{[s]}^+(\mu_1\dots\mu_k)}-\xi_{\lambda_{[s]}^-(\mu_1\dots\mu_k)}\right).
\end{equation*}
\end{lemma}

\begin{proof}
It is a computation using \eqref{eq:partialderivativesxi}, as well as the definitions of $\lambda_{[s]}^+$ and $\lambda_{[s]}^-$. Indeed,
\begin{eqnarray*}
\frac{\partial}{\partial x^{\mu_k}}\xi_{\lambda(\mu_1\dots\mu_{k-1})}=\frac{\partial}{\partial x^{\mu_k}}\prod_{i=1}^l\xi_{\lambda_i(\mu_1\dots\mu_{k-1})}\\
=\sum_{s=1}^l\left(\left(\prod_{i=1}^{s-1}\xi_{\lambda_i(\mu_1\dots\mu_{k-1})}\right)\left(\frac{\partial}{\partial x^{\mu_k}}\xi_{\lambda_s(\mu_1\dots\mu_{k-1})}\right)\left(\prod_{i=s+1}^l\xi_{\lambda_i(\mu_1\dots\mu_{k-1})}\right)\right)\\
=\sum_{s=1}^l\left(\left(\prod_{i=1}^{s-1}\xi_{\lambda_i(\mu_1\dots\mu_{k-1})}\right)\left(\xi_{(\lambda_s\cup\{k\})(\mu_1\dots\mu_k)}-\xi_{\lambda_s(\mu_1\dots\mu_{k-1})}\xi_{\mu_k}\right)\left(\prod_{i=s+1}^l\xi_{\lambda_i(\mu_1\dots\mu_{k-1})}\right)\right)\\
=\sum_{s=1}^l\left(\xi_{\lambda_{[s]}^+(\mu_1\dots\mu_k)}-\xi_{\lambda_{[s]}^-(\mu_1\dots\mu_k)}\right).
\end{eqnarray*}
\end{proof}

We are ready to compute the coordinated expression of the injective morphism $J^k(X,G)\hookrightarrow G\times\bigoplus_{j=1}^k\left(\bigotimes^j T^*X\otimes\mathfrak g\right)$.

\begin{theorem}\label{theorem:explicitxi}
For each $k\in\mathbb Z^+$, the components of the map \eqref{eq:injectionhigherorderjet} with $\nabla$ being the canonical flat connection on the trivial bundle $T^*X=\mathbb R^n\times\mathbb R^n$ are given by
\begin{equation}\label{eq:componentsJk}
\xi^{(j)}(x)=\sum_{\lambda\in\mathcal P_1^+(j)}\epsilon(\lambda)~\xi_{\lambda(\mu_1\dots\mu_j)}(x)~dx^{\mu_1}\otimes\dots\otimes dx^{\mu_j},\qquad 1\leq j\leq k,
\end{equation}
where $\epsilon:\mathcal P_1^+(j)\to\{-1,1\}$ is defined by recurrence as follows: $\epsilon(\{1\})=1$, $\epsilon(\lambda_{[s]}^+)=\epsilon(\lambda)$ and $\epsilon(\lambda_{[s]}^-)=-\epsilon(\lambda)$ for each $j\in\mathbb Z^+$, $\lambda\in\mathcal P_1^+(j)$ and $1\leq s\leq j$.
\end{theorem}

\begin{proof}
Again, we show the result by induction in $k$. For $k=1$, we have
\begin{equation*}
\xi^{(1)}(x)=(dg)_x g(x)^{-1}=(\partial g/\partial x^\mu)(x)g(x)^{-1}dx^\mu=\xi_\mu(x)dx^\mu.    
\end{equation*}
Note that this expression agrees with \eqref{eq:componentsJk} since $\mathcal P_1^+(1)=\{(\{1\})\}$ and $\epsilon(\{1\})=1$.

To conclude, given $k>1$, we assume that the result holds for $k-1$. Hence,
\begin{eqnarray*}
\xi^{(k)}(x)=\left(d\xi^{(k-1)}\right)_x=d\left(\sum_{\lambda\in\mathcal P_1^+(k-1)}\epsilon(\lambda)\xi_{\lambda(\mu_1\dots\mu_{k-1})}dx^{\mu_1}\otimes\dots\otimes dx^{\mu_{k-1}}\right)_x\\
=\sum_{\lambda\in\mathcal P_1^+(k-1)}\epsilon(\lambda)\left(\frac{\partial}{\partial x^{\mu_k}}\xi_{\lambda(\mu_1\dots\mu_{k-1})}\right)(x) dx^{\mu_1}\otimes\dots\otimes dx^{\mu_k}\\
=\sum_{\lambda\in\mathcal P_1^+(k-1)}\sum_{s=1}^l\epsilon(\lambda)\left(\xi_{\lambda_{[s]}^+(\mu_1\dots\mu_k)}(x)-\xi_{\lambda_{[s]}^-(\mu_1\dots\mu_k)}(x)\right)dx^{\mu_1}\otimes\dots\otimes dx^{\mu_k}\\
=\sum_{\lambda\in\mathcal P_1^+(k)}\epsilon(\lambda)\xi_{\lambda(\mu_1\dots\mu_k)}(x)dx^{\mu_1}\otimes\dots\otimes dx^{\mu_k},
\end{eqnarray*}
where we have used Proposition \ref{prop:disjointunionpartly} and Lemma \ref{lemma:partialxi}.
\end{proof}

\begin{remark}
As a straightforward consequence of this theorem, we deduce that the maps $\xi_{\mu_1\dots\mu_j}$ take values in the universal enveloping algebra of $\mathfrak g$. More specifically, for each $j\in\mathbb Z^+$, $\xi_{\mu_1\dots\mu_j}:X\to U_j(\mathfrak g)$, where we define
\begin{equation*}
U_j(\mathfrak g)=\mathfrak g+\textrm{\normalfont span}\left\{B_{i_1}\dots B_{i_j}\mid(i_1,\dots,i_j)\in\mathcal S_j\right\}\subset U(\mathfrak g),
\end{equation*}
the sum being the one of vector subspaces. Note that the sum becomes direct sum, $\oplus$, for $j>1$, whereas $U_1(\mathfrak g) = \mathfrak g + \mathfrak g = \mathfrak g$. In any case, despite expression \eqref{eq:componentsJk}, the elements $\xi^{(j)}$ lie in $\mathfrak{g}$ by construction.

\end{remark}

Lastly, we use the previous theorem to compute the image of the map \eqref{eq:injectionhigherorderjet} for $k=2$. We denote by
\begin{equation*}
\textrm{\normalfont Sym}:\bigotimes^k T^*X\to\bigvee^k T^*X,\qquad \textrm{\normalfont Skew}:\bigotimes^k T^*X\to\bigwedge^k T^*X    
\end{equation*}
the symmetrization and skew-symmetrization maps, respectively. Recall that for $k=2$ these maps induce the decomposition $T^*X\otimes T^*X=(T^*X\vee T^*X)\oplus(T^*X\wedge T^*X)$. Similarly, for each $\xi_1,\xi_2\in U(\mathfrak g)$ we denote by $[\xi_1,\xi_2]=\xi_1\xi_2-\xi_2\xi_1$ the commutator, and by $\{\xi_1,\xi_2\}=\xi_1\xi_2+\xi_2\xi_1$ the anti-commutator.

\begin{corollary}
We have that 
\begin{equation*}
J^2(X,G)\simeq G\times V_2\subset G\times(T^*X\otimes\mathfrak g)\oplus(T^*X\otimes T^*X\otimes\mathfrak g),
\end{equation*}
where $V_2$ is the affine sub-bundle defined as
\begin{equation*}
V_2=\bigg\{(\xi_x,\alpha_x)\in(T^*X\otimes\mathfrak g)\oplus(T^*X\otimes T^*X\otimes\mathfrak g)\mid
\textrm{\normalfont Skew}(\alpha_x)=-\frac{1}{2}[\xi_x,\xi_x]\bigg\}.
\end{equation*}
In particular, $\textrm{\normalfont Sym}(\alpha_x)=\xi_{\mu\nu}(x)dx^\mu\otimes dx^\nu-\frac{1}{2}\{\xi_x,\xi_x\}.$ 
\end{corollary}

\begin{proof}
Since the map \eqref{eq:injectionhigherorderjet} is an injective morphism, it is enough to show that its image is, precisely, $G\times\left((T^*X\otimes\mathfrak g)\oplus V_2\right)$. Thank to the previous theorem, we know that \eqref{eq:injectionhigherorderjet} is given by
\begin{equation*}
j_x^2 g\mapsto\left(g(x),\xi^{(1)}(x)=\xi_\mu(x) dx^\mu,\xi^{(2)}(x)=(\xi_{\mu\nu}(x)-\xi_\mu(x)\xi_\nu(x))dx^\mu\otimes dx^\nu\right),
\end{equation*}
where $\xi_\mu(x)=(\partial g/\partial x^\mu)(x)g(x)^{-1}$ and $\xi_{\mu\nu}(x)=(\partial^2 g/\partial x^\mu\partial x^\nu)(x)g(x)^{-1}$. It is clear that the first and second components of the map are surjective onto $G$ and $T^*X\otimes\mathfrak g$, respectively. 
Additionally, it is easy to check that
\begin{equation*}
\textrm{\normalfont Skew}\left(\xi^{(2)}(x)\right)=-\frac{1}{2}\left[\xi^{(1)}(x),\xi^{(1)}(x)\right],
\end{equation*}
Analogously,
\begin{equation*}
\textrm{\normalfont Sym}\left(\xi^{(2)}(x)\right)=\xi_{\mu\nu}(x)dx^\mu\otimes dx^\nu-\frac{1}{2}\left\{\xi^{(1)}(x),\xi^{(1)}(x)\right\},
\end{equation*}
and we conclude.
\end{proof}

\section*{Acknowledgments}

MCL and ARA have been partially supported by Ministerio de Ciencia e Innovaci\'on (Spain), under grants PGC2018-098321-B-I00 and PID2021-126124NB-I00.  

ARA has been supported by a FPU grant from Ministerio de Universidades (Spain).

\bibliographystyle{plainnat}
\bibliography{biblio.bib}

\begin{thebibliography}{8}
\providecommand{\natexlab}[1]{#1}
\providecommand{\url}[1]{\texttt{#1}}
\expandafter\ifx\csname urlstyle\endcsname\relax
  \providecommand{\doi}[1]{doi: #1}\else
  \providecommand{\doi}{doi: \begingroup \urlstyle{rm}\Url}\fi

\bibitem[Colombo and de~Diego(2014)]{CoMa2014}
L.~Colombo and D.~Martín de~Diego.
\newblock Higher-order variational problems on {L}ie groups and optimal control
  applications.
\newblock \emph{Journal of Geometric Mechanics}, 6\penalty0 (4):\penalty0
  451--478, 2014.

\bibitem[Esen et~al.(2021)Esen, Kudeyt, and Sütlü]{EsKuSu2021}
O.~Esen, M.~Kudeyt, and S.~Sütlü.
\newblock Second order {L}agrangian dynamics on double cross product groups.
\newblock \emph{Journal of Geometry and Physics}, 159:\penalty0 103934, 2021.
\newblock ISSN 0393-0440.
\newblock \doi{https://doi.org/10.1016/j.geomphys.2020.103934}.
\newblock URL
  \url{https://www.sciencedirect.com/science/article/pii/S0393044020302217}.

\bibitem[Gay-Balmaz et~al.(2012)Gay-Balmaz, Holm, Meier, Ratiu, and
  Vialard]{GaHoMeRaVi2012}
F.~Gay-Balmaz, D.~Holm, D.~Meier, T.~Ratiu, and F.-X. Vialard.
\newblock Invariant higher-order variational problems.
\newblock \emph{Commun. Math. Phys.}, 309:\penalty0 413--458, 2012.
\newblock \doi{10.1007/s00220-011-1313-y}.

\bibitem[Hall(2015)]{Ha2015}
B.~Hall.
\newblock \emph{Lie Groups, {L}ie Algebras, and Representations}.
\newblock Graduate Texts in Mathematics. Springer, Cham, second edition, 2015.
\newblock \doi{https://doi.org/10.1007/978-3-319-13467-3}.

\bibitem[Mangiarotti and Sardanashvily(2000)]{MaSa2000}
L.~Mangiarotti and G.~Sardanashvily.
\newblock \emph{Connections in Classical and Quantum Field Theory}.
\newblock World Scientific, 2000.
\newblock \doi{10.1142/2524}.
\newblock URL \url{https://www.worldscientific.com/doi/abs/10.1142/2524}.

\bibitem[Saunders(1989)]{Sa1989}
D.~J. Saunders.
\newblock \emph{The geometry of jet bundles}.
\newblock Cambridge University Press, 1989.

\bibitem[Varadarajan(1984)]{Va1984}
V.~S. Varadarajan.
\newblock \emph{Lie Groups, {L}ie Algebras, and Their Representations}.
\newblock Graduate Texts in Mathematics. Springer, New York, NY, 1 edition,
  1984.
\newblock ISBN 978-1-4612-1126-6.
\newblock \doi{https://doi.org/10.1007/978-1-4612-1126-6}.

\bibitem[Vizman(2013)]{Vi2013}
C.~Vizman.
\newblock The group structure for jet bundles over {L}ie groups.
\newblock \emph{Journal of Lie Theory}, 23\penalty0 (3), 2013.

\end{thebibliography}

\end{document}